\documentclass{article}
\usepackage[utf8]{inputenc}
\usepackage{enumerate}
\usepackage{amsmath,amssymb,amsthm,amsfonts,amstext}
\usepackage{mathtools}
\usepackage[english]{babel}
\usepackage{csquotes}
\usepackage{geometry}
\usepackage{graphicx}
\usepackage{bibentry}
\usepackage{xcolor}
\usepackage{bm}

\allowdisplaybreaks

\usepackage[textsize=tiny]{todonotes}
\theoremstyle{definition}
\newtheorem{definition}{Definition}

\theoremstyle{remark}

\theoremstyle{plain}

\newtheorem{proposition}[definition]{Proposition}

\renewcommand*{\phi}{\varphi}
\renewcommand*{\theta}{\vartheta}
\renewcommand*{\rho}{\varrho}

\renewcommand*{\Gamma}{\varGamma}
\renewcommand*{\Delta}{\varDelta}
\renewcommand*{\Theta}{\varTheta}
\renewcommand*{\Lambda}{\varLambda}
\renewcommand*{\Xi}{\varXi}
\renewcommand*{\Pi}{\varPi}
\renewcommand*{\Sigma}{\varSigma}
\renewcommand*{\Upsilon}{\varUpsilon}
\renewcommand*{\Phi}{\varPhi}
\renewcommand*{\Psi}{\varPsi}
\renewcommand*{\Omega}{\varOmega}

\newcommand*{\RR}{\ensuremath{\mathbb{R}}}

\newcommand*{\NN}{\ensuremath{\mathbb{N}}}

\DeclareMathOperator{\GW}{GW}
\DeclareMathOperator{\GM}{GM}

\newcommand*{\T}{\mathrm{T}}

\AtBeginDocument{\colorlet{defaultcolor}{.}}
\title{On Assignment Problems
Related to 
\\
Gromov--Wasserstein Distances
on the Real Line}
\author{Robert Beinert\footnote{Institute of Mathematics,
    Technische Universit\"at Berlin,
    Stra\ss{}e des 17. Juni 136,
    10623 Berlin, Germany 
   (\mbox{beinert@math.tu-berlin.de},
    \mbox{steidl@math.tu-berlin.de}).},
  Cosmas Heiss\footnotemark[1],
  Gabriele Steidl\footnotemark[1]}

\begin{document}
\maketitle

\begin{abstract}
  Let  $x_1 < \ldots < x_n$ and $y_1 < \ldots < y_n$, $n \in \mathbb N$, be real numbers.
	We show by an example that the assignment problem
	$$
  \max_{\sigma \in S_n}
  F_\sigma(x,y)
  \coloneqq
  \frac12
  \sum_{i,k=1}^n
  |x_i- x_k|^\alpha \, |y_{\sigma(i)}- y_{\sigma(k)}|^\alpha, \quad \alpha >0,
$$
is in general neither solved by the identical permutation ($\text{id}$) nor the anti-identical permutation ($\text{a-id}$)
if $n > 2 +2^\alpha$. 
Indeed the above maximum can be, depending  on the number of points,
arbitrary far away from $F_\text{id}(x,y)$ and $F_\text{a-id}(x,y)$.
The motivation to deal with such assignment problems came from their relation to 
Gromov-Wasserstein divergences which have recently attained a lot of attention. 
\end{abstract}

\section{Introduction}
The Gromov--Wasserstein (GW) distance
as combination of the Gromov--Hausdorff and the Wasserstein distance
has been introduced by M\'emoli \cite{Mem2011}
in order to measure the distance between metric measure spaces.
This distance enjoys great popularity in the machine learning community
since it allow the comparison of probability measures
living on spaces of different dimensions.
Further,
the GW distance is invariant under isometries like shifts and rotations,
which is desirable in certain application.
Unfortunately,
the computation of the GW distance requires the minimization
of a non-convex quadratic program,
which is numerically challenging and time-consuming.
As a remedy  the so-called sliced GW distance \cite{VFC+2019,NDJBS2021}
has recently attracted much attention in the scientific community.
It has similar properties as the Gromov--Wasserstein distance
but appears to be superior with respect to the numerical implementation.
The central ingredients are the Gromov--Wasserstein distance on the real line and
it's rearrangement to an assignment problem.
Indeed, this important relation, which we will briefly explain below, was our motivation to deal with the topic.
Numerical studies raise the conjecture that the assignment problem could be solved by the
identical or anti-identical permutation.
The contribution of this small note is to show that this is in general not the case.

\section{Gromov--Wasserstein and Assignment Problems in $\mathbb R$} \label{sec:GW1}
Let $c \colon \RR \times \RR \to [0,\infty)$
be some symmetric cost function such that $c(t,t) = 0$ for $t \in \RR$,
and let $\mu \coloneqq \sum_{i=1}^n p_i \delta_{x_i}$ and
$\nu \coloneqq \sum_{j=1}^m q_i \delta_{y_j}$
be two discrete probability measures
with pairwise distinct real-valued $x_i$, resp. $y_i$, $i=1,\ldots,n$.
The non-negative weights $p \coloneqq (p_i)_{i=1}^n$ and $q \coloneqq (q_j)_{j=1}^m$
here satisfy $\bm 1^\T p = 1$ and $\bm 1^\T q = 1$.
The \emph{Gromov--Wasserstein distance} (on the line) is defined as
\begin{equation*}
  \GW(\mu, \nu)
  \coloneqq
  \min_{\pi \in \Pi(p,q)}
  \sum_{i,j=1}^{n,m}
  \sum_{k, \ell =1}^{n,m}
  | c(x_i, x_k) - c(y_j, y_\ell)|^2 \,
  \pi_{i,j} \pi_{k, \ell},
\end{equation*}
where $\Pi(p,q)$ denotes all matrices $\pi \coloneqq (\pi_{i,j})_{i,j=1}^{n,m}$
with $\pi \bm 1 = p$ and $\bm 1^\T \pi  = q^\T$.
The optimal GW plan $\pi$ describes
how much mass is transported from $x_i$ to $y_j$.
For $n=m$,
$p_i = q_j = 1/n$, and
$x_1 < \cdots < x_n$
as well as
$y_1 < \dots < y_n$,
we may instead look for an optimal (one-to-one) GW map
$\sigma \colon \{1,\dots,n\} \to \{1, \dots, n\}$
minimizing
\begin{equation*}
  \GM(\mu, \nu)
  \coloneqq
  \min_{\sigma \in S_n}
  \frac1{n^2}
  \sum_{i=1}^{n}
  \sum_{k=1}^{n}
  | c(x_i, x_k) - c(y_{\sigma(i)}, y_{\sigma(k)})|^2,
\end{equation*}
where $S_n$ denotes the permutation group of $\{1, \dots, n\}$.
Each optimal GW map $\sigma$ corresponds to a maybe non-optimal GW plan $\pi$
via $\pi_{i,j} = 1$ for $j = \sigma(i)$ and $\pi_{i,j} = 0$ otherwise.
The mass is thus completely moved from $x_i$ to $y_{\sigma(i)}$.
This problem is also known as the \emph{Gromov--Monge (GM)}.
Clearly, GW and GM are closely related.

Up to $1/n^2$,
the objective of GM may be rearranged as
\begin{align*}
  &\sum_{i,k=1}^{n}
  | c(x_i, x_k) - c(y_{\sigma(i)}, y_{\sigma(k)})|^2
  \\
  &\qquad=
    \sum_{i,k=1}^n
    \Bigl[
    c^2(x_i, x_k)
    - 2 c(x_i, x_k) \, c(y_{\sigma(i)}, y_{\sigma(k)})
    + c^2(y_{\sigma(i)}, y_{\sigma(k)})
    \Bigr]
  \\
  &\qquad=
    \sum_{i,k=1}^n
    c^2(x_i, x_k)
    - 2 \sum_{i,k=1}^n
    c(x_i, x_k) \, c(y_{\sigma(i)}, y_{\sigma(k)})
    + \sum_{i,k=1}^n
    c^2(y_{i}, y_{k}).
\end{align*}
Since the first and last sum are independent of $\sigma$,
finding a minimizer of GM is equivalent to
finding a maximizer of the \emph{Assignment Problem} 
\begin{equation}   \label{eq:quad-ass}
  \max_{\sigma \in S_n}
  F_\sigma(x,y)
  \qquad\text{with}\qquad
  \begin{aligned}[t]
  F_\sigma(x,y)
  &\coloneqq
  \frac12
  \sum_{i,k=1}^n
  c(x_i, x_k) \, c(y_{\sigma(i)}, y_{\sigma(k)})
  \\
  &=
  \sum_{\substack{i,k=1\\i<k}}^n
  c(x_i, x_k) \, c(y_{\sigma(i)}, y_{\sigma(k)})
\end{aligned}
\end{equation}
and $x = (x_i)_{i=1}^n$, $y = (y_j)_{j=1}^n$
with ascending ordered pairwise distinct components.
We are interested in the special
cost functions
$c(s,t) \coloneqq |s - t|^\alpha$
with $\alpha > 0$.
The case $\alpha = 1$
corresponding to the metric $d(t,s) \coloneqq |t-s|$
is of special interest
since it is related to the classic GW distance on the line.
Although numerical experiments may indicate that the maximizer of
the assignment problem \eqref{eq:quad-ass} is either the
identity $\mathrm{id}(i) \coloneqq i$
or by the anti-identity $\text{a-id}(i) \coloneqq n - i + 1$ on $\{1,\ldots,n\}$,
the following proposition shows that this is in general not the case.

\begin{proposition}
  \label{prop:1}
  Let $c(s,t) \coloneqq |s - t|^\alpha$, $\alpha > 0$,
  Then there exist $n > 2 + 2^\alpha$ and $x,y \in \RR^n$ with ascending ordered pairwise distinct components 
	such that
  \begin{equation*}
    F_{\rm{id}}(x,y)
    < \max_{\sigma \in S_n} F_\sigma(x,y)
    \qquad\text{and}\qquad
    F_{\text{\upshape a-id}}(x,y)
    < \max_{\sigma \in S_n} F_\sigma(x,y) .
		\end{equation*}
		Moreover, the gap can become arbitrary large for increasing $n \in \NN$.
 \end{proposition}

\begin{proof} Assuming that the maximizer is always given by the identity od anti-identity,
we prove the assertion by a counterexample.
  For given $n>3$ and $\alpha > 0$,
  we construct an explicite instance
  by studying
  $x(\epsilon) = (x_i)_{i=1}^n$ and
  $y(\epsilon) = (y_i)_{i=1}^n$
  with $\epsilon \in (0, 2/(n-3))$
  given by
  \begin{equation*}
    x_i \coloneqq
    \begin{cases}
      -1,
      & i = 1,
      \\
      \tfrac{2i-n-1}{2} \, \epsilon,
      & i = 2, \dots, n-1,
      \\
      1,
      & i = n
    \end{cases}
    \qquad\text{and}\qquad
    y_i \coloneqq
    \begin{cases}
      -1,
      & i = 1,
      \\
      -1 + \epsilon,
      & i = 2,
      \\
      (i - 2) \, \epsilon,
      & i=3, \dots, n.
    \end{cases}
  \end{equation*}
  Due to the antisymmetry $x_i = - x_{n-i+1}$,
  we have
  $F_{\mathrm{id}}(x(\epsilon), y(\epsilon))
  = F_{\text{a-id}}(x(\epsilon), y(\epsilon))$.
	Let $f_\sigma(\epsilon) \coloneqq F_{\sigma}(x(\epsilon), y(\epsilon))$, $\sigma \in S_n$.
  Then, considering the summands with $i,k=1,2,n$ in \eqref{eq:quad-ass} separately,   we obtain
  \begin{align*}
    f_{\mathrm{id}}(\epsilon)
     &=
      \epsilon^{2 \alpha}
      \sum_{\substack{i,k = 3\\i < k}}^{n-1}
      | i- k |^{2\alpha}
      +
      \epsilon^\alpha
      \sum_{i=3}^{n-1}
      \bigl| \tfrac{2i-n-1}2 \, \epsilon - 1 \bigr|^\alpha \,
      |i-n|^\alpha
		+
      \epsilon^\alpha \,
      \bigl| \tfrac{3-n}2 \, + 1 \bigr|^\alpha
			\\
    &\qquad+
		\epsilon^\alpha
      \sum_{k=3}^{n-1}
      | 2 - k |^\alpha \,
      | (k-3) \, \epsilon + 1  |^\alpha
			+
      \sum_{k=3}^{n-1}
      \bigl| \tfrac{2k-n-1}2 \, \epsilon + 1 \bigr|^\alpha \,
      |(k-2) \, \epsilon  + 1 |^\alpha
          \\
    &\qquad+
      2^\alpha \, | (n-2) \, \epsilon + 1 |^\alpha
      +
      \bigl|\tfrac{3-n}2 \, \epsilon - 1 \bigr|^\alpha \,
      | (n - 3) \, \epsilon  + 1|^\alpha.      
  \end{align*}
  Next, we consider the cyclic permutation $\sigma = \mathrm{cyc}$ given by
  \begin{equation*}
    \mathrm{cyc}(i) \coloneqq
    \begin{cases}
      i + 1,
      & i=1,\dots, n-1,
      \\
      1,
      & i = n,
    \end{cases}
		\end{equation*}
  which gives
		\begin{equation*}
    y_{\mathrm{cyc}(i)} \coloneqq
    \begin{cases}
      -1 + \epsilon,
      & i = 1,
      \\
      (i - 1) \, \epsilon,
      & i=2, \dots, n-1,
      \\
      -1,
      & i = n.
    \end{cases}
  \end{equation*}
  Individual consideration of the summands $i,k = 1,n$ in \eqref{eq:quad-ass} yields 
  \begin{align*}
    f_{\mathrm{cyc}}(\epsilon)
           &=
      \epsilon^{2\alpha}
      \sum_{\substack{i,k = 2\\i < k}}^{n-1}
      | i-k |^{2\alpha}
      + 
			2^\alpha \, \epsilon^\alpha
			+
      \sum_{k=2}^{n-1}
      \bigl| \tfrac{2k -n -1}2 \, \epsilon + 1 \bigr|^\alpha \,
      |(k-2) \, \epsilon + 1|^\alpha
    \\
           &\qquad+ 
             \sum_{i=2}^{n-1}
             \bigl|\tfrac{2i-n-1}2 \, \epsilon - 1 \bigr|^\alpha \,
             |(i-1) \, \epsilon + 1|^\alpha.
  \end{align*}
  Evaluating both functions for the degenerate case $\epsilon = 0$,
  we get 
  $f_{\mathrm{id}}(0) = 2^\alpha + (n-2) $ 
  and
  $ f_{\mathrm{cyc}}(0) =     2(n-2)
  $, so that
  \begin{equation*}
    f_{\mathrm{cyc}}(0) - f_{\mathrm{id}}(0)
    =
    (n-2) - 2^\alpha  > 0 \quad \text{if }  n > 2+ 2^\alpha.
  \end{equation*}
  Due to the continuity of $f_{\mathrm{id}}$ and $f_{\mathrm{cyc}}$ in $\epsilon$,
  there exists an $\epsilon > 0$ such that
  $f_{\mathrm{id}}(\epsilon) < f_{\mathrm{cyc}}(\epsilon)$ and the difference can become arbitrary large for increasing $n$.
  Although the cyclic permutation may be no maximizer by itself,
  we obtain the assertion.
\end{proof}

Our counterexample consists of an artificial point arrangement.
Numerical study suggests that
the maximizer is often given by either $\mathrm{id}$ or $\text{a-id}$.
How high the probability in fact is, remains open for future research.


\bibliographystyle{abbrv}
\bibliography{references}

\end{document}